\documentclass[11pt]{article}
\usepackage{stmaryrd}
\usepackage{amsmath}
\usepackage{amsfonts}
\usepackage{txfonts}
\usepackage{hyperref}
\usepackage{graphicx,amssymb,amsfonts,amsmath,amscd}
\usepackage[all,ps,cmtip]{xy}
\usepackage{amscd}
\usepackage{xspace} % To get the right spacings in front of : and so on
\usepackage{mathrsfs}

\def\bC{{\bf C}}

\def\bN{{\bf N}}

\def\bQ{{\bf Q}}

\def\bZ{{\bf Z}}

\def\C{\bC}

\def\Z{\bZ}
\def\N{\bN}
\def\Q{\bQ}

\def\sO{{\mathscr O}}

\def\sW{{\mathscr W}}

\def\sZ{{\mathscr Z}}

\def\ie{\textit {i.e.} }
\def\cf{\textit {cf.} }
\def\vs{\textit {vs.} }
\def\loccit{\textit {loc.cit.} }

\def\apriori{\textit{a priori} }

\def\etc{\textit {etc.} }

\def\Bl{\mathop{\rm Bl}\nolimits} %Blow-up
\def\CH{\mathop{\rm CH}\nolimits} % Chow groups
\def\Hom{\mathop{\rm Hom}\nolimits}
\def\id {\mathop{\rm id}\nolimits} %identity
\def\im{\mathop{\rm Im}\nolimits} % image, imaginary part is \Im
\def\Ker{\mathop{\rm Ker}\nolimits} % kernel
\def\pr{\mathop{\rm pr}\nolimits} % projection
\def\prim{\mathop{\rm prim}\nolimits} %primitive part

\def\rank{\mathop{\rm rank}\nolimits}
\def\Supp {\mathop{\rm Supp}\nolimits} % Support
\def\miu{\mathop{\bf \mu} \nolimits} % roots of unity

\def\bar{\overline}
\def\inj{\hookrightarrow}
\def\surj{\twoheadrightarrow}

\def\lra{\xrightarrow}

\def\cart{\ar@{}[dr]|\square} % cartesian diagrams, write it after the left-up term to produce a square in the middle of the diagram
\def\cisom{\cong} % canonical isomorphism
\def\isom{\simeq} %isomorphism
\def\tilde{\widetilde}

\newenvironment{proof}{\trivlist \item[\hskip\labelsep{\sc{
Proof.}}]}{\hbox to.1pt{\hss}\hfill$\square$\bigskip\endtrivlist}

\newenvironment{prooflemma}{\trivlist \item[\hskip\labelsep{\sc
Proof of the Lemma.}]\rm}{\hbox
to.1pt{\hss}\hfill$\square$\bigskip\endtrivlist}

\newtheorem{thm}{Theorem}[section]
\newtheorem{prop}[thm]{Proposition}
\newtheorem{lemma}[thm]{Lemma}
\newtheorem{cor}[thm]{Corollary}
\newtheorem{conj}[thm]{Conjecture}

\newtheorem{defi}[thm]{Definition}
\newtheorem{rmk}[thm]{Remark}
\newtheorem{rmks}[thm]{Remarks}
\newtheorem{expl}[thm]{Example}
\newtheorem{def-prop}[thm]{Definition-Proposition}
\newtheorem{prop-def}[thm]{Proposition-Definition}

\begin{document}

\title{On the Coniveau of Certain Sub-Hodge Structures}
\date{ }
\author{Lie Fu\\
D\'epartement de Math\'ematiques et Applications\\ \'Ecole Normale Sup\'erieure\\45 Rue d'Ulm,
75230 Paris Cedex 05\\France\\ lie.fu@ens.fr}
\maketitle

\begin{abstract}
We study the generalized Hodge conjecture for certain sub-Hodge structure defined as the kernel of the cup product map with a big cohomology class, which is of Hodge coniveau at least 1. As predicted by the generalized Hodge conjecture, we prove that the kernel is supported on a divisor, assuming the Lefschetz standard conjecture.
\end{abstract}

\section{Introduction}
Given a smooth projective variety $X$ defined over $\C$, let $H^k(X,\Q)$ be its $k$-th Betti cohomology group, which carries a pure Hodge structure of weight $k$. We can ask the philosophical question: how much information about the geometry of the variety, for example a knowledge of its subvarieties, can be extracted from the shape of certain associated transcendental objects, namely the Hodge structures on its cohomology groups? The generalized Hodge conjecture formulates a precise such relationship.

Recall that the \emph{Hodge coniveau} of a weight $k$ (pure) Hodge structure $(L, L^{p,q})$ is defined to be the largest integer $c\leq\lfloor\frac{k}{2}\rfloor$ such that $L^{0,k}=L^{1,k-1}=\cdots=L^{c-1,k-c+1}=0$. If for any integer $c$, we define $N^c_{\text{Hdg}}\!H^k(X,\Q)$ as the sum of all the sub-Hodge structures in $H^k(X,\Q)$ of Hodge coniveau at least $c$, we obtain the \emph{Hodge coniveau filtration} on $H^k(X,\Q)$. On the other hand, in terms of the topology of algebraic subvarieties of $X$, we also have the so-called \emph{arithmetic filtration} or \emph{coniveau filtration} $N^cH^k(X,\Q)$ (\cf \cite{MR0244271} \cite{MR0252404} \cite{MR0412191}), where $N^cH^k(X,\Q)$ consists of the cohomology classes supported on some algebraic subset of codimension at least $c$, here \emph{supported} on a closed subset means the class becomes zero when it is restricted to the open complement. The following inclusion (\cf \S \ref{sectionGHC}) gives a first relation between the two filtrations:
\begin{equation}\label{inclusion}
 N^cH^k(X,\Q)\subset N^c_{\text{Hdg}}\!H^k(X,\Q).
\end{equation}
In his famous paper \cite{MR0252404}, Grothendieck conjectures that the two filtrations in fact coincide, more precisely:
\begin{conj}[Grothendieck amended generalized Hodge conjecture]\label{GHC}
Let $X$ be a smooth projective variety of dimension $n$, $0\leq k\leq n$ be an integer, and $L\subset H^k(X, \Q)$ be a sub-Hodge structure of Hodge coniveau at least $c$, then there exists a closed algebraic subset $Z$ of codimension at least $c$, such that $$L\subset \Ker\left(j^*: H^k(X,\Q)\to H^k(X\backslash Z, \Q)\right),$$ where $j: X\backslash Z\to X$ is the natural inclusion.
\end{conj}
Note that the usual Hodge conjecture is the case $k=2c$.\\
The usual Hodge conjecture already has many theoretical consequences. For example, it implies that a morphism of Hodge structures between the cohomology groups of two smooth projective varieties is always induced by an algebraic correspondence (\cf Remark \ref{morphism}). In particular, it implies the Lefschetz standard conjecture (\cf \S\ref{lefsection}), which says that the inverse of the hard Lefschetz isomorphism $$L_X^i: H^{n-i}(X,\Q)\lra{\isom} H^{n+i}(X,\Q)$$
is induced by an algebraic cycle in $\CH^{n-i}(X\times X)_\Q$.
The generalized Hodge conjecture has strong implications about the Chow groups, let us just mention \cite{MR1205880} \cite{MR2585582} \cite{HodgeBloch}.

The usual Hodge conjecture is widely open. The known cases of it include $k=2c=0,2,2n-2,2n$ (thus for varieties of dimension at most 3), varieties with cellular decomposition (Grassmannians, flag varieties, or more generally, quotients of reductive linear algebraic groups by parabolic subgroups), cubic 4-folds (\cite{MR0453741} \cite{MR818549}) \etc    While for the generalized Hodge conjecture, besides the aforementioned cases, very few are known. One class of known cases concerns about algebraic varieties with an automorphism group, see for example \cite{MR1133322} and \cite{MR1205880}. As far as we know, besides these and some results about abelian varieties (\cf \cite{MR1485922}  \cite{MR1831630} \cite{MR1923969}), there are no general results verifying the conjecture for a proper sub-Hodge structure.

In this paper, we try to understand such a sub-Hodge structure situation. Our starting point is a discovery of a sub-Hodge structure of Hodge coniveau $\geq1$, which we describe here in the case of divisors for simplicity.

For an ample divisor $A$ on an $n$-dimensional smooth projective variety $X$, the hard Lefschetz isomorphism tells us in particular that $\Ker(\cup [A]: H^{n-1}(X)\to H^{n+1}(X))$ vanishes. Now if we  weaken the positivity assumption, namely consider a \emph{big} divisor $D=A+E$, where $A$ is an ample divisor, and $E=\sum_im_iE_i$ is an effective divisor, then in general, $$L:=\Ker\left(\cup [D]: H^{n-1}(X)\to H^{n+1}(X)\right)$$ could be non-trivial, for instance: (see also Example \ref{example})\\

\textbf{Example.} Let $X=\Bl_yY\lra{\tau} Y$ be the blow-up of a point in a smooth projective 3-fold $Y$, and $D:=\tau^*(\sO_Y(1))$ be the pull back of an ample divisor on $Y$. Then $D$ is big, while $L=\Ker(\cup [D]: H^{2}(X)\to H^{4}(X))$ is generated by the fundamental class of the exceptional divisor $[E]\in H^2(X)$.
\vspace{0.3cm}

Although $L$ does not vanish in general, we still expect the positivity condition on $D$ implies some control on $L$. In the above example we can readily see that $L$ is supported on a divisor, thus of Hodge coniveau $\geq 1$ in particular. Following the idea of \cite{MR2585582}, we get in general:\\

\textbf{Observation (Lemma \ref{observation}):} $L$ is of Hodge coniveau at least 1.
\vspace{0.4cm}

Indeed, for any class $\alpha\in H^{n-1,0}(X)$ (in particular it is primitive), if $D\cup\alpha=0$ but $\alpha\neq 0$, then
$ 0 = \int_X D\alpha\bar\alpha
= \int_X [A] \alpha\bar\alpha + \int_X [E]\alpha\bar\alpha
= \int_X [A]\alpha\bar\alpha + \sum_i m_i\int_{\tilde E_i}\tau_i^*(\alpha)\bar{\tau_i^*(\alpha)}
$,
where $\tau_i:\tilde E_i\to E_i$ is a resolution of singularities for each $i$.
However the second Hodge-Riemann bilinear relation (\cf \cite{MR1967689}) gives
\begin{equation*}
 (-1)^{\frac{n(n-1)}{2}}i^{n-1}\int_X [A]\alpha\bar\alpha>0
\end{equation*}
and for any $i$
\begin{equation*}
 (-1)^{\frac{n(n-1)}{2}}i^{n-1}\int_{\tilde E_i}\tau_i^*(\alpha)\bar{\tau_i^*(\alpha)}\geq 0
\end{equation*}
Summing up these inequalities, we have a contradiction, therefore $\alpha=0$, thus proving the observation.

Regarding the generalized Hodge conjecture, we ask the natural\\

\textbf{Question (Conjecture \ref{mainquestion}).} Can we prove that the kernel $L$ of cup product with a big class is supported on a divisor of $X$?
\vspace{0.3cm}

We answer this question in this paper assuming the Lefschetz standard conjecture. Here is the main theorem, where the role of big divisor classes is played by the more general notion of \emph{big cohomology classes} (\cf Definition \ref{bigness}):
\begin{thm}[=Theorem \ref{mainthm}]\label{mainthm0}
 Let $X$ be a smooth projective variety of dimension $n$, $k\in \{0, 1,\cdots, n\}$ be an integer, and $\gamma\in H^{2n-2k}(X,\Q)$ be a big cohomology class. Let $L$ be the kernel of the following morphism of `cup product with $\gamma$':
$$\cup\gamma: H^k(X,\Q)\to H^{2n-k}(X,\Q).$$
Then assuming the Lefschetz standard conjecture, $L$ is supported on a divisor of $X$, that is, $$L\subset \Ker\left(H^k(X,\Q)\to H^k(X\backslash Z, \Q)\right)$$ for some closed algebraic subset $Z$ of codimension 1.
\end{thm}

The proof consists of three steps:
\begin{itemize}
 \item Proposition \ref{geometrization} realizes $L(1)$ effectively as a sub-Hodge structure of the degree $(k-2)$ cohomology of some other smooth projective variety. This step reduces the question to the usual Hodge conjecture;
 \item We use the standard conjecture to construct adjoint correspondences (\S\ref{adjointsubsection}) to get a divisor-supported sub-Hodge structure which is transverse to the orthogonal complement of $L$ as in (\ref{surj2});
 \item We use the adjoint correspondences (\S\ref{adjointsubsection}) to construct the orthogonal projector onto $L$ in Proposition \ref{projectorL}.
\end{itemize}

Here is the structure of this paper.
In \S2, besides some general remarks on the generalized Hodge conjecture, we give a description of the gap between the usual and the generalized Hodge conjectures (Lemma \ref{gap}). In \S3 we introduce the coniveau 1 sub-Hodge structure mentioned above, which is our main object of study, and we show that the generalized Hodge conjecture \ref{GHC} is satisfied for it assuming the usual Hodge conjecture. In \S4 we begin by making some general remarks concerning the Lefschetz standard conjecture, then we give the basic construction of the so-called \emph{adjoint correspondences}, and finally we prove our main theorem \ref{mainthm0}. In the last section \S5, we discuss some unconditional results and give a reinterpretation of our main result in the language of motivated cycles of Y. Andr\'e.\\

\section*{Acknowledgements} 
I would like to express my gratitude to my thesis advisor Claire Voisin for bringing to me this interesting question, as well as her kind help and great patience during this work.

\section{Generalities of the Generalized Hodge Conjecture}\label{sectionGHC}
In this section, we introduce the generalized Hodge conjecture and make a comparison with the usual Hodge conjecture.

First of all, let us recall some standard terminologies\footnote{We will ignore the usual factor $2\pi i$, which makes the formulations in algebraic de Rham cohomology and in Betti cohomology compatible. But we will not make any such comparison argument in this paper.}. Let $m$ be an integer.
\begin{itemize}
 \item The \emph{Tate Hodge structure} $\Q(m)$ is the pure Hodge structure of weight $-2m$, with the underlying rational vector space $\Q$, and with the Hodge decomposition concentrated at bidegree $(-m,-m)$.
 \item The \emph{Tate twist} $L(m)$ of a pure Hodge structure $L$ of weight $k$ is defined to be the tensor product $L\otimes \Q(m)$, which is a Hodge structure of weight $k-2m$. More concretely, the underlying rational vector space is $L$, while the Hodge decomposition $L(m)\otimes_\Q \C=\oplus_{p+q=k-2m}L(m)^{p,q}$ is given by $L(m)^{p,q}=L^{p+m,q+m}$.
 \item A weight $k$ pure Hodge structure $(L, L_\C=\oplus_{p+q=k}L^{p,q})$ is called \emph{effective}, if $L^{p,q}=0$ when $p<0$ or $q<0$.
\end{itemize}

Here is the important notion of \emph{Hodge coniveau} of a Hodge structure.

\begin{defi}[Hodge coniveau]\upshape
Let $$(L, L_\C=\bigoplus_{\substack{p+q=k \\ p,q\geq 0}}L^{p,q})$$ be an effective pure Hodge structure of weight $k$. The \emph{Hodge coniveau} of $L$ is defined to be the largest integer $c$ such that the Tate twist $L(c)$ is an effective pure Hodge structure of weight $k-2c$. In other words, the Hodge decomposition takes the following form
$$L_\C=L^{c,k-c}\oplus L^{c+1, k-c-1}\oplus\cdots\oplus L^{k-c,c}$$
with $L^{c,k-c}\neq 0$.

Note that the Hodge coniveau of a non-zero effective pure Hodge structure of weight $k$ is always $\leq\lfloor\frac{k}{2}\rfloor$.
\end{defi}

Given a smooth projective variety $X$ of dimension $n$, and a closed algebraic subset $Z$ of codimension $\geq c$, it is an easy consequence of the strictness of morphisms between mixed Hodge structures (\cf \cite{MR0498551}) that
$$\Ker\left(H^k(X)\lra{j^*} H^k(X\backslash Z)\right)= \im\left(H_{2n-k}(Z)(-n)\lra{i_*} H^k(X)\right)$$ is equal to
$$\im \left(H_{2n-k}(\tilde Z)(-n)\lra{\tilde i_*} H^k(X)\right)=\im\left(H^{k-2c}(\tilde Z)(-c)\lra{\tilde i_*} H^k(X)\right),$$
where $i, j$ are the inclusions, $\tau: \tilde Z\to Z$ is a resolution of singularities of $Z$, $\tilde i=i\circ \tau$, and all the (co-)homology groups are with rational coefficients. Since $H^{k-2c}(\tilde Z,\Q)(-c)$ is a Hodge structure of Hodge coniveau $\geq c$, we deduce that $\im\left(\tilde i_*\right)$ hence $\Ker(j^*)$ is a sub-Hodge structure of Hodge coniveau at least $c$.

This explains in particular the inclusion (\ref{inclusion}) in the introduction: $$N^cH^k(X,\Q)\subset N^c_{\text{Hdg}}\!H^k(X,\Q),$$ while Grothendieck's generalized Hodge conjecture \ref{GHC} states the reverse inclusion. In the situation of the conjecture, we say that $L$ is \emph{supported on $Z$}.
\begin{rmks}\upshape The generalized Hodge conjecture is widely open.
 \begin{itemize}
  \item The cases $k=0, 1$ are trivial, and the case of $k=2$ follows from the Lefschetz theorem on (1,1)-classes.
  \item For $k\leq n$, if we view the cohomology group $H^{2n-k}(X,\Q)$ as of weight $k$ via the twist $\Q(n-k)$, the analogous conjecture for $H^{2n-k}(X,\Q)$ follows from the conjecture for $H^k(X,\Q)$ by hard Lefschetz isomorphisms.
  \item Note that the usual Hodge conjecture is exactly the case when $k=2i$ and $c=i$, since to give a sub-Hodge structure of Hodge coniveau $i$ in $H^{2i}(X,\Q)$ amounts to give a Hodge class of degree $2i$ up to a constant scalar. The known cases of the usual Hodge conjecture include $k=2c=0,2,2n-2,2n$ (thus for varieties of dimension at most 3), varieties with cellular decomposition (Grassmannians, flag varieties), cubic 4-folds (\cite{MR0453741} \cite{MR818549}) and so on.
  \item For general complete intersections in projective spaces, the generalized Hodge conjecture for the middle cohomology is equivalent to the generalized Bloch conjecture, assuming the Lefschetz standard conjecture (\cf \cite{HodgeBloch}).
  \item \cite{MR1205880} deals with some complete intersection surfaces with an automorphism group. See also \cite{MR1133322} for a similar result about Calabi-Yau 3-folds.
  \item There are some results for abelian varieties (\cf \cite{MR1485922}  \cite{MR1831630} \cite{MR1923969}).
 \end{itemize}
\end{rmks}

We would like to make the following well-known remark which says that the gap between the usual Hodge conjecture and the generalized Hodge conjecture is the problem of looking for an \emph{effective realization} of the Tate twist of the sub-Hodge structure. For more general remarks to the generalized Hodge conjecture, we refer to the papers \cite{MR894051} \cite{MR715646}.

\begin{lemma}[Hodge conjecture \vs Generalized Hodge conjecture]\label{gap}
   Let $X$ be a smooth projective variety of dimension $n$, and $L\subset H^k(X,\Q)$ be a sub-Hodge structure of Hodge coniveau at least $c$. Assume the following condition:\\
$(*)$ There exists a smooth projective variety $Y$, such that $L(c)$ is a sub-Hodge structure of $H^{k-2c}(Y,\Q)$.\\
Then the usual Hodge conjecture for $Y\times X$ implies the generalized Hodge conjecture for $L$.
\end{lemma}
Before the proof of the lemma, let us recall the following fundamental interpretation of a morphism between two Hodge structures as a Hodge class in their $\Hom$-space viewed as a Hodge structure (\cf \cite{MR1967689}):

\begin{rmk}\label{morphism}\upshape
Let $k_1, k_2\in \Z$ be of the same parity, and we set $c=\frac{k_2-k_1}{2}\in \Z$. Let $L_1$, $L_2$ be two rational pure Hodge structures of weights $k_1, k_2$ respectively. The canonical identification $\Hom_{\Q}(L_1, L_2)= L_1^*\otimes_\Q L_2$ induces on $\Hom_{\Q}(L_1, L_2)$ a Hodge structure of weight $k_2-k_1$. Then a linear map $f\in \Hom_{\Q}(L_1, L_2)$ is a morphism of Hodge structures of bidegree $(c,c)$ if and only if $f$ is a Hodge class of degree $2c$ with respect to this natural Hodge structure. In the geometric setting, let $X$, $Y$ be smooth projective varieties of dimension $n, m$ respectively, and $f: H^{k_1}(X,\Q)\to H^{k_2}(Y,\Q)$ be a $\Q$-linear map, then $f$ is a morphism of Hodge structures of bidegree $(c,c)$ if and only if $f$ is a Hodge class of degree $2c$ in $\Hom_{\Q}\left(H^{k_1}(X,\Q), H^{k_2}(Y,\Q)\right)=H^{k_1}(X,\Q)^*\otimes_\Q H^{k_2}(Y,\Q)\cisom H^{2n-k_1}(X,\Q)(n)\otimes_\Q H^{k_2}(Y,\Q)$, which is a direct factor of $H^{2n-k_1+k_2}(X\times Y,\Q)(n)$ by the K\"unneth formula. For such Hodge class $f$, if moreover
there is an algebraic cycle $\sZ\in \CH^{n+c}(X\times Y)_\Q$ such that the fundamental class $[\sZ]\in H^{2n+2c}(X\times Y,\Q)$ coincides with $f$ when projecting to the K\"unneth factor $H^{2n-k_1}(X,\Q)(n)\otimes_\Q H^{k_2}(Y,\Q)$, then we say that $f$ is \emph{algebraic}, meaning that $f$ is induced by an algebraic correspondence. In particular, the (usual) Hodge conjecture implies that any morphism of Hodge structures $f: H^{k_1}(X,\Q)\to H^{k_2}(Y,\Q)$ is in fact algebraic.
\end{rmk}

\begin{prooflemma} (\cf \cite{MR894051}.)
 Since the Hodge structure $H^{k-2c}(Y, \Q)$ is polarizable, $L(c)$ is a direct factor of $H^{k-2c}(Y, \Q)$ in the category of Hodge structures. In particular, there is a projection  $H^{k-2c}(Y, \Q)\surj L(c)$, which is a morphism of Hodge structures. Twisting it by $\Q(-c)$, and composing with the inclusion of $L$ into $H^k(X,\Q)$, we get a morphism of Hodge structures $$H^{k-2c}(Y, \Q)(-c)\to H^k(X,\Q)$$ with image $L$.
Now apply the usual Hodge conjecture for $Y\times X$ (\cf Remark \ref{morphism}), we conclude that this morphism of Hodge structures is algebraic, \ie it is the correspondence induced by an algebraic cycle $\sZ\in \CH_{n-c}(Y\times X)$. Therefore, $L=\im\left([\sZ]_*:H^{k-2c}(Y)(-c)\to H^{k}(X)\right)$ is supported on $Z:= \Supp\left(\pr_2(\sZ)\right)$, the support of the image of $\sZ$ under the projection to $X$. Clearly, every irreducible component of $Z$ is of dimension at most $\dim(\sZ)=n-c$, hence of codimension at least $c$.
\end{prooflemma}

\begin{rmks}\upshape
The condition $(*)$ in the above lemma is always satisfied when $k=2c$ (trivial) or $k=2c+1$ (thanks to the anti-equivalence of categories between weight 1 effective rational Hodge structures and abelian varieties up to isogenies). Moreover, by the Lefschetz theorem of hyperplane sections, we can reduce to the case of $\dim(Y)=k-2c$ by taking successive general hyperplane sections on $Y$. 
\end{rmks}

\section{Kernel of the Cup Product Map with Big Classes}
% In this section, we first give an interesting example of a sub-Hodge structure of Hodge coniveau at least 1, namely the kernel of the morphism of `cup product with a big cohomology class', from which we come to our main question (Conjecture \ref{mainquestion}), and finally we show that it is implied by the usual Hodge conjecture, while in \S\ref{sectionproof} we will deduce it from the Lefschetz standard conjecture.
%
% Let $X$ be a smooth projective variety of dimension $n$. Recall that a divisor $D$ on $X$ is called \emph{big}, if\footnote{In fact for big divisors it is a limit for $m$ sufficiently divisible.}
% $$\limsup_{m\to \infty}\frac{\dim H^0(X, \sO_X(mD))}{m^n}>0.$$
% By a lemma of Kodaira (\cf \cite{MR2095471}), the bigness of $D$ is equivalent to the property that for any ample divisor $A$, there exists some multiple of $D$ which is linear equivalent to the sum of the ample divisor $A$ and an effective divisor $E$: $$mD\equiv_{\lin} A+E,$$ and conversely any divisor with a multiple numerically equivalent to the sum of an ample divisor and an effective divisor is always big.
%
% Unfortunately, for algebraic cycles of higher codimension, the theory of positivity has not yet been well developed, we make the following nonstandard definition which at least coincides in the divisor case with the standard notion of bigness defined above.

For a smooth projective variety $X$, let $H^{2i}(X,\Q)_{\text{alg}}$ be the $\Q$-subspace of $H^{2i}(X,\Q)$ generated by the fundamental classes of algebraic cycles of codimension $i$. In $H^{2i}(X,\Q)_{\text{alg}}$ sits the \emph{effective cone} generated by the effective algebraic cycles of codimension $i$. Making an analogue of the divisor case, we define a cohomology class to be \emph{big}, if it is in the interior (when passing to the real coefficients) of the effective cone. Here is the practical definition that we will use in this paper.

\begin{defi}[Big cohomology classes]\label{bigness}\upshape
Let $X$ be a smooth projective variety, and let $0\leq i\leq \dim(X)$ be an integer. A cohomology class $\gamma\in H^{2i}(X, \Q)(i)$ is called \emph{big}, if some of its positive multiples is of the form
$$m\gamma=[A]^i+[E] ~~\text{in } H^{2i}(X, \Z)(i), ~~ m\in \N^*,$$
where $A$ is an ample divisor, $E$ is an effective algebraic cycle of codimension $i$, and [-] means the cohomology class of an algebraic cycle.
\end{defi}

To simplify the notation, we will mostly suppress the Tate twists from now on, except when we want to highlight it.

Note that if the class $\gamma\in H^{2i}(X, \Q)$ is `ample' in the sense that $\gamma\in [A]^i\cdot \Q_{>0}$ for some ample divisor $A$, then the hard Lefschetz theorem says $\cup \gamma: H^{n-i}(X,\Q)\to H^{n+i}(X,\Q)$ is an isomorphism; in particular, the kernel is trivial. But when $\gamma$ is only big, the kernel could be non-trivial as the following example shows.

\begin{expl}\label{example}\upshape
 Let $V$ be a smooth projective variety of dimension $n$ with a smooth subvariety $Z$ of codimension $c\geq 2$. Let $X:=\Bl_ZV\lra{\tau} V$ be the blow-up of $V$ along $Z$, and $E$ be the exceptional divisor:
\begin{displaymath}
 \xymatrix{E \cart\ar[r]^{\iota} \ar[d]_{p}
  & \Bl_ZV \ar[d]^{\tau}\\
Z\ar[r]_{i} & V}
\end{displaymath}
We consider $\gamma=\tau^*(A)$, the pull-back of an ample divisor
class $A$ on $V$. Thanks to the following formula for the cohomology of blow-ups (\cf \cite{MR1967689} Theorem 7.31):
\begin{equation*}
 \tau^*\oplus\bigoplus_{i=0}^{c-2}\iota_*\xi^ip^*: H^{n-1}(V)\oplus\bigoplus_{i=0}^{c-2}H^{n-3-2i}(Z)\lra{\isom}H^{n-1}(X)
\end{equation*}
where $\xi=\sO_E(1)$, we find that $$\Ker\left(\cup\gamma: H^{n-1}(X)\to H^{n+1}(X)\right)\isom\bigoplus_{i=0}^{c-2}\Ker\left(\cup A|_Z: H^{n-3-2i}(Z)\to H^{n-1-2i}(Z)\right),$$ which does not vanish in general.
\end{expl}

Despite $\Ker(\cup \gamma)\neq 0$ in general, we still expect the positivity assumption on $\gamma$ would imply the kernel is `small' in certain sense. For instance in the above example, we observe that the kernel is in fact supported in the exceptional divisor $E$; in particular, $\Ker(\cup \gamma)$ is of Hodge coniveau at least 1.

The following Lemma \ref{observation} generalizes this example. This observation is the starting point of the paper. The idea of using the Hodge-Riemann bilinear relations goes back to \cite{MR2585582}. 
\begin{lemma}[Observation] \label{observation}
Let $X$ be a smooth projective variety of dimension $n$, $0\leq k\leq n$ be an integer, and $\gamma\in H^{2n-2k}(X,\Q)$ be a big cohomology class. Let $L$ be the kernel of the following morphism of `cup product with $\gamma$':
$$\cup\gamma: H^k(X,\Q)\to H^{2n-k}(X,\Q).$$
Then $L$ is a sub-Hodge structure of $H^k(X,\Q)$ of Hodge coniveau at least 1.
\end{lemma}
\begin{proof}
  Replacing $\gamma$ by a multiple if necessary, we can suppose that $$\gamma=[A]^{n-k}+[E]~~\text{in } H^{2n-2k}(X, \Z),$$
where $A=\sO_X(1)$ is a general hyperplane section, and $E$ is an effective algebraic cycle of codimension $n-k$.
Since $\cup \gamma$ is clearly a morphism of Hodge structures, its kernel $L$ is of course a sub-Hodge structure. Therefore to prove the Hodge coniveau 1 assertion, which means $L^{k,0}=0$, it suffices to show that for any class $\alpha\in H^{k,0}(X)$, if $\gamma\cup\alpha=0$, then $\alpha=0$. Let $\alpha$ be such a class.

Let $E=\sum_im_iE_i$ with $m_i\in \N^*$ be the decomposition into linear combination of prime divisors, and $\tau_i:\tilde E_i\to E_i$ be a resolution of singularities for each $i$. As $\gamma \cup \alpha=0$, we have
\begin{eqnarray*}
 0 &=& \int_X \gamma\alpha\bar\alpha \\
&=& \int_X [A]^{n-k}\alpha\bar\alpha + \int_X [E]\alpha\bar\alpha\\
&=& \int_X [A]^{n-k}\alpha\bar\alpha + \sum_i m_i\int_{\tilde E_i}\tau_i^*(\alpha)\bar{\tau_i^*(\alpha)}
\end{eqnarray*}
However, since $\alpha$ is primitive in $H^k(X,\C)$, by the second Hodge-Riemann bilinear relation (\cf \cite{MR1967689}),
\begin{equation}\label{positive1}
 (-1)^{\frac{k(k-1)}{2}}i^k\int_X [A]^{n-k}\alpha\bar\alpha\geq0,
\end{equation}
with equality holds only when $\alpha=0$.

Similarly, since $\tau_i^*(\alpha)$ is also of type $(k,0)$, in particular primitive in $H^k(\tilde E_i, \C)$, we have again by the second Hodge-Riemann bilinear relation that for each $i$,
\begin{equation}\label{positive2}
 (-1)^{\frac{k(k-1)}{2}}i^k\int_{\tilde E_i}\tau_i^*(\alpha)\bar{\tau_i^*(\alpha)}\geq 0.
\end{equation}
As the sum of the left hand sides of (\ref{positive1}) and (\ref{positive2}) is zero, we have an equality in (\ref{positive1}), \ie $\alpha=0$, and hence $L$ is of Hodge coniveau at least 1.
\end{proof}

Combining the above observation \ref{observation} with the generalized Hodge conjecture \ref{GHC}, one gets the following conjecture which is the main subject of the paper.

\begin{conj}\label{mainquestion}
 Let $X$ be a smooth projective variety of dimension $n$, $0\leq k\leq n$ be an integer, and $\gamma\in H^{2n-2k}(X,\Q)$ be a big cohomology class (in the sense of Definition \ref{bigness}). Let $L$ be the kernel of the following morphism of `cup product with $\gamma$':
$$\cup\gamma: H^k(X,\Q)\to H^{2n-k}(X,\Q).$$
Then $L$ is supported on a divisor of $X$, \ie $L\subset \Ker\left(H^k(X,\Q)\to H^k(X\backslash Z, \Q)\right)$ for some $Z$ closed algebraic subset of codimension 1.
\end{conj}

In the presence of Lemma \ref{observation}, the cases of $k=0, 1, n$ are trivial, and the case of $k=2$ follows from the Lefschetz theorem on (1,1)-classes.

We would like to show first (see Corollary \ref{byHodge}) that Conjecture \ref{mainquestion} is implied by the usual Hodge conjecture. The key point is the following Proposition \ref{geometrization} of effective realization of $L(1)$. The argument appeared in C. Voisin's paper \cite{MR2585582}. We reproduce her argument here since the construction will be useful in \S\ref{sectionproof}, where we will show that Conjecture \ref{mainquestion} is in fact a consequence of the Lefschetz standard conjecture.

\begin{prop}[Effective realization]\label{geometrization}
 Let $X$, $k$, $\gamma$, $L$ be as above. Then there exists a (not necessarily connected) smooth projective variety $Y$ of dimension $k-1$ with a morphism $\miu: Y\to X$, such that the composition $L\inj H^k(X, \Q)\lra{\miu^*}H^k(Y, \Q)$ is injective.\\
 In particular, $L(1)$ is a sub-Hodge structure of $H^{k-2}(Y, \Q)$.
\end{prop}
\begin{proof}
 Adopting the notations in Lemma \ref{observation}, up to replacing $\gamma$ by a positive multiple, we can assume $$\gamma=[A]^{n-k}+[E]~~\text{in } H^{2n-2k}(X, \Z),$$
where $A=\sO_X(1)$ is a general hyperplane section and $E=\sum_im_iE_i$ with $m_i\in \N^*$ is an effective algebraic cycle of dimension $k$, and $\tau_i:\tilde E_i\to E_i$ be a resolution of singularities for each $i$. Let $B$ be the intersection of $(n-k+1)$ general hyperplane sections of $X$, and $H_i$ be a general section of a very ample line bundle on $\tilde E_i$, in particular, $B$ and $H_i$ are irreducible smooth projective varieties of dimension $k-1$.

Let $Y:=B\sqcup\bigsqcup_iH_i$ be their disjoint union, and $\miu: Y\to X$ be the natural morphism. We claim:\\
$(**)$ The composition $L\inj H^k(X,\Q)\lra{\miu^*}H^k(Y,\Q)$ is injective, \ie $L\cap \Ker(\miu^*)=\{0\}$.

Indeed, since $L\cap \Ker(\miu^*)$ is a sub-Hodge structure, it suffices to show, for each $(p,q)$ with $p+q=k$, that if $\alpha\in H^{p,q}(X)$ satisfies $\alpha\cup \gamma=0$ and $\miu^*(\alpha)=0$, then we have $\alpha=0$. Suppose the contrary: $\alpha\neq 0$.

Since the composition  $H^k(X,\Q)\lra{i_B^*} H^k(B,\Q)\lra{{i_B}_*}H^{2n-k+2}(X,\Q)$ is exactly the Lefschetz operator $[B]=[A]^{n-k+1}$ and the second morphism is an isomorphism by Lefschetz's hyperplane theorem, we find that $\Ker\left(i_B^*:H^k(X,\Q)\to H^k(B,\Q)\right)=H^k(X,\Q)_{\prim}$, where $i_B=\miu|_B$ is the natural inclusion of $B$ into $X$; in particular, $\alpha$ is a primitive class of type $(p,q)$.

As in the proof of Lemma \ref{observation}, firstly we have
\begin{eqnarray*}
 0 &=& \int_X \gamma\alpha\bar\alpha \hspace{0.8cm}\text{(since $\gamma \cup \alpha=0$)}\\
&=& \int_X [A]^{n-k}\alpha\bar\alpha + \int_X [E]\alpha\bar\alpha\\
&=& \int_X [A]^{n-k}\alpha\bar\alpha + \sum_i m_i\int_{\tilde E_i}\tau_i^*(\alpha)\bar{\tau_i^*(\alpha)}
\end{eqnarray*}
However, since $\alpha\neq 0$ is primitive of type $(p,q)$, by the second Hodge-Riemann bilinear relation (\cf \cite{MR1967689}), we have
\begin{equation*}
 (-1)^{\frac{k(k-1)}{2}}i^{p-q}\int_X [A]^{n-k}\alpha\bar\alpha>0.
\end{equation*}
Therefore, since the sum is zero, there exists $i$, such that
\begin{equation*}
 (-1)^{\frac{k(k-1)}{2}}i^{p-q}\int_{\tilde E_i}\tau_i^*(\alpha)\bar{\tau_i^*(\alpha)}<0.
\end{equation*}
By the second Hodge-Riemann bilinear relation of $\tilde E_i$, we deduce that $\tau_i^*(\alpha)$ is NOT primitive in $H^k(\tilde E_i,\Q)$, \ie $[H_i]\cup \tau_i^*(\alpha)\neq 0$. In particular, $(\miu|_{H_i})^*(\alpha)\neq 0$, giving a contradiction to the assumption that $\alpha\in \Ker(\miu^*)$. So the claim $(**)$ follows, and this is exactly what we want.

As for the last assertion, composing the injective morphism of Hodge structures obtained above $L\inj H^k(Y, \Q)$, with the inverse of the hard Lefschetz isomorphism (as Hodge structures) $H^{k-2}(Y,\Q)(-1)\lra{\isom} H^k(Y,\Q)$, we get an inclusion of Hodge structures $L(1)\inj H^{k-2}(Y,\Q)$ as desired.
\end{proof}

\begin{cor}\label{byHodge}
 Conjecture \ref{mainquestion} is implied by the usual Hodge conjecture.
\end{cor}
\begin{proof}
To reach the generalized Hodge conjecture from the usual Hodge conjecture, we use Lemma \ref{gap} which explains the gap between them, so we only have to check in our situation the condition $(*)$ in Lemma \ref{gap}.

However, the above Proposition \ref{geometrization} provides an inclusion of Hodge structures $L(1)\inj H^{k-2}(Y,\Q)$, and this is exactly the condition $(*)$ in Lemma \ref{gap}.
\end{proof}

The rest of the paper is devoted to the proof of the main theorem \ref{mainthm0}, which says that Conjecture \ref{mainquestion} is in fact implied by an \apriori much weaker conjecture, namely the Lefschetz standard conjecture.

\section{Lefschetz Standard Conjecture implies Conjecture \ref{mainquestion}}\label{sectionproof}

We first recall the Lefschetz standard conjecture. Then in the second subsection we deal with the construction and the formal properties of the \emph{adjoint} of an algebraic correspondence, which incorporates the strength of the Lefschetz standard conjecture; while in the third subsection, by rather formal arguments, we will deduce our main theorem \ref{mainthm} from  Proposition \ref{geometrization}, which embeds the Tate twist of the sub-Hodge structure in question into the cohomology of some smooth projective variety.

\subsection{The Lefschetz Standard Conjecture}\label{lefsection}
Here we gather some well-known general remarks concerning the Lefschetz standard conjecture, for a more complete treatment, see \cite{MR0292838} \cite{MR1265519}.
Let $X$ be a smooth projective variety of dimension $n$, $\sO_X(1)$ be a very ample divisor which is chosen to be the polarization of $X$.
Let $\xi=c_1(\sO_X(1))\in H^2(X,\Q)$. Define the Lefschetz operator $$L_X=\cup \xi : H^k(X,\Q)\to H^{k+2}(X,\Q)$$ to be cup product with the first Chern class of the polarization.  The hard Lefschetz theorem asserts that for any integer $k\in \{0, \cdots,n\}$, the morphism
$$L_X^{n-k}:H^k(X,\Q)\to H^{2n-k}(X,\Q)$$ is an isomorphism.
Note that this isomorphism is in fact \emph{algebraic} (see Remark \ref{morphism}), which means that it is the correspondence induced by a dimension $k$ algebraic cycle ${\Delta_X}_*(\sO_X(1)^{n-k})\in \CH_k(X\times X)$, where $\Delta_X:X\inj X\times X$ is the diagonal inclusion.
In his paper \cite{MR0268189}, Grothendieck conjectures that the inverse of the Lefschetz isomorphism is also algebraic.
\begin{conj}[Lefschetz standard conjecture]\label{standard}
In the above situation, there exists a codimension $k$ algebraic cycle with rational coefficients $\sZ\in \CH^k(X\times X)_\Q$, such that the induced correspondence
$$[\sZ]_*: H^{2n-k}(X,\Q)\to H^k(X,\Q)$$ is the inverse of the isomorphism $L_X^{n-k}$ defined above.
\end{conj}

\begin{rmks}\upshape\label{RkStand}
We list some basic facts about the standard conjecture. Some of them will be used in the sequel.
\begin{itemize}
%  \item We call any one of $B(X)$, $\theta(X)$, $\nu(X)$ and ${}^{\prim} C(X)$ the \emph{Lefschetz standard conjecture}, and we call $C(X)$ the \emph{K\"unneth standard conjecture}.
 \item There are several equivalent versions of the Lefschetz standard conjecture (\cf \cite{MR0292838} \cite{MR1265519}). Besides the one stated above, let us just mention another equivalent one which says that the projectors $\pi_{L_X^i H^{k-2i}(X)_{\prim}}$, with respect to the Lefschetz decomposition $H^k(X)=\bigoplus_{i\geq \max\{0,k-n\}}L_X^i H^{k-2i}(X)_{\prim}$, are algebraic.
 \item The Lefschetz standard conjecture implies the K\"unneth standard conjecture which says that all the projectors $\pi^k:  H^*(X)\surj H^k(X)\inj H^*(X)$ are algebraic.

 \item The Lefschetz standard conjecture is implied by the usual Hodge conjecture. Indeed, $(L_X^{n-k})^{-1}$ is a morphism of Hodge structures, by Remark \ref{morphism} the corresponding
cohomology class of $X\times X$ is a Hodge class (it is an \emph{absolute} Hodge class\footnote{Roughly speaking, they `descend' with the field of definition of $X$, \cf \cite{MR654325}.} in fact), and the Hodge conjecture claims the existence of an algebraic cycle inducing $(L_X^{n-k})^{-1}$.
 \item The Lefschetz standard conjecture in degree 1, namely the algebraicity of $(L^{n-1})^{-1}: H^{2n-1}(X,\Q)\to H^1(X,\Q)$, is implied by the Lefschetz theorem of (1,1)-class on $X\times X$.
Thus the Lefschetz standard conjecture is known for curves and surfaces. Besides, other known cases include abelian varieties, generalized flag varieties. Note that this conjecture is stable by taking products, hyperplane sections (\cf \cite{MR1265519}). Let us also mention the recent work \cite{StandConjHK} verifying this conjecture for certain type of irreducible holomorphic symplectic varieties.
\end{itemize}
\end{rmks}

\subsection{Adjoint correspondences}\label{adjointsubsection}
For any smooth projective variety $X$ of dimension $n$, with polarization $\sO_X(1)$ and corresponding Lefschetz operator $L_X$, let us consider the following operator $s_X$ on $H^*(X,\Q)$, which changes the signs of the factors in the Lefschetz decomposition to retain the positivity property as in the primitive part.

\begin{defi}[Operator $s_X$ on $H^*(X)$]\upshape
For any integer $k\in \{0, \cdots,n\}$,
the action of the operator $s_X$ on $H^k(X)=\bigoplus_{0\leq i\leq \lfloor\frac{k}{2}\rfloor}L_X^i H^{k-2i}(X)_{\prim}$ is defined as multiply by $(-1)^{\frac{k(k-1)}{2}}\cdot(-1)^i$ on the direct factor $L_X^i H^{k-2i}(X)_{\prim}$ in the Lefschetz decomposition. Let the action of $s_X$ on $H^{2n-k}(X)$ be the action induced from the one on $H^k(X)$ via the hard Lefschetz isomorphism.
\end{defi}
\begin{rmks}\label{rmk-s}\upshape From the above definition, we note that
\begin{itemize}
 \item $s_X$ is an involution: $s_X\circ s_X=\id$;
 \item $s_X$ commutes with the Lefschetz operator $L_X\circ s_X=s_X\circ L_X$;
 \item The transpose of $s_X$ is $s_X$;
 \item $s_X$ is \emph{rational}, \ie it comes from a $\Q$-linear operator on $H^*(X,\Q)$, the reason is that the Lefschetz decomposition is rational.
\end{itemize}

\end{rmks}

\begin{lemma}[Algebraicity of $s_X$]
 Assuming the Lefschetz standard conjecture, the operator $s_X$ is algebraic, \ie it is induced by an algebraic cycle in $\CH^n(X\times X)_\Q$.
\end{lemma}
\begin{proof}
 We can write $$s_X=\sum_{k=0}^{n}(\sum_{i=0}^{\lfloor\frac{k}{2}\rfloor}(-1)^i\pi_{L_X^i H^{k-2i}(X)_{\prim}})\cdot(-1)^{\frac{k(k-1)}{2}}\pi^k+ \sum_{k=0}^{n}(\sum_{i=0}^{\lfloor\frac{k}{2}\rfloor}(-1)^i\pi_{L_X^{n-k+i} H^{k-2i}(X)_{\prim}})\cdot(-1)^{\frac{k(k-1)}{2}}\pi^{2n-k}.$$
By the first two points of Remark \ref{RkStand}, assuming the Lefschetz standard conjecture, all the projectors appearing in the above formula, hence $s_X$ itself, are algebraic.
\end{proof}

Usually, for $0\leq k\leq n$, people use the pairing $<-,->$ on $H^k(X)$ defined by
$$<x,y>:=\int_XL_X^{n-k}xy ~~~~~\text{for any $x,y\in H^k(X)$}$$
Although it is non-degenerate thanks to the hard Lefschetz isomorphism, it does not have the positivity property enjoyed by the primitive part any more. In the language of Hodge theory, we say that this pairing is NOT a polarization. To retain the positivity property,
we define the following modified bilinear pairing on $H^k(X)$:
\begin{equation}\label{polarization}
(x,y)_{H^k(X)}:= \int_XL_X^{n-k}x\cdot s_X(y)
\end{equation}
for any $x,y \in H^k(X)$. We sometimes suppress the subscript to write $(-,-)$ if we don't want to mention the Hodge structure explicitly. Then
\begin{itemize}
 \item $(-,-)_{H^k(X)}$ is rational;
 \item $(x,y)_{H^k(X)}=(-1)^k(y,x)_{H^k(X)}$ ~for any $x, y\in H^k(X)$;
\end{itemize}
Moreover, by the Hodge-Riemann bilinear relations (\cf \cite{MR1967689}), we find that for any $0\leq k\leq n$, any $x\in H^{p,q}(X)$ and $y\in H^{p',q'}(X)$ with $p+q=p'+q'=k$, we have
\begin{itemize}
 \item $(x,y)=\int_XL_X^{n-k}x\cdot s_X(y)=0$ unless $(p,q)=(q',p')$;
 \item $(i^{p-q}x,\bar x)=\int_XL_X^{n-k}\cdot i^{p-q}x\cdot s_X(\bar x)>0$ for any $0\neq x\in H^{p,q}(X)$.
\end{itemize}
Therefore the bilinear pairing (\ref{polarization}) on $H^k(X)$ is a \emph{polarization} (\cf \cite{MR1967689}) of the Hodge structure $H^k(X,\Q)$. As for the Hodge structure $H^{2n-k}(X,\Q)$, we use the polarization induced from the one on $H^k(X,\Q)$ via the hard Lefschetz isomorphism.

Throughout this paper, we will always use this polarization (\ref{polarization}) on the cohomology groups of any polarized smooth projective variety.

\begin{rmk}\label{advantage}\upshape
The advantage of using the polarizations $(-,-)$ instead of the usual pairings $<-,->$ can be summarized in the following very vague analogue: \emph{as long as we stay\footnote{That is, all the vector spaces considered are Hodge structures and all the relevant morphisms between them are morphisms of Hodge structures. In particular, all the subspaces should be sub-Hodge structures.} in the category of polarizable Hodge structures equipped with the polarizations above, to do linear algebra we can pretend that the spaces are euclidean spaces equipped with positive definite scalar products.} To illustrate this intuition as well as for later use, we  want to recall here several basic properties of polarizations of Hodge structures, and more analogues can be found in the rest of this paper. Let $H$ be a Hodge structure with polarization $(-,-)$, and $L$ be a sub-Hodge structure, then (\cf \cite{MR1967689})
\begin{itemize}
 \item $(-,-)|_L$ gives a polarization of $L$;
 \item $L^\perp$ is a sub-Hodge structure with polarization $(-,-)|_{L^\perp}$;
 \item $L\cap L^\perp=\{0\}$, thus $L\oplus L^\perp=H$.
\end{itemize}
\end{rmk}

Here comes the basic terminology that we will use in the following.
\begin{prop-def}[Adjoint correspondence]\label{adjoint}
Let $X$, $Y$ be smooth projective varieties of dimension $n$, $m$ respectively, and $-n\leq r\leq m$ be an integer. Given $\sZ\in \CH^{n+r}(X\times Y)_\Q$ an algebraic cycle with rational coefficients, viewed as a correspondence (\cf \cite{MR1644323}) from $X$ to  $Y$, it induces morphisms on cohomology groups for any $k\in \{0,\cdots,2n\}$: $$C:=[\sZ]_*: H^k(X,\Q)\to H^{k+2r}(Y,\Q).$$
Assuming the Lefschetz standard conjecture, then there exists an algebraic cycle with coefficients $\sZ^\dagger\in \CH^{m-r}(Y\times X)_\Q$, such that as a correspondence from $Y$ to $X$, for any $k\in \{0,\cdots,2n\}$, the induced morphism on cohomology groups:
$$C^\dagger:= [\sZ^\dagger]_*: H^{k+2r}(Y,\Q)\to H^k(X,\Q)$$
satisfies
\begin{equation}\label{adjointeqn}
 (C\alpha, \beta)_{H^{k+2r}(Y)}=(\alpha, C^\dagger\beta)_{H^k(X)}
\end{equation}
for any $\alpha\in H^k(X)$ and any $\beta\in H^{k+2r}(Y)$, where $(-,-)$ denotes the polarization of Hodge structures fixed in (\ref{polarization}).\\
We call $\sZ^\dagger$ an \emph{adjoint correspondence} of $\sZ$, and also $C^\dagger$ the \emph{adjoint (cohomological) correspondence} of $C$.
\end{prop-def}
\begin{proof}
 Since the Lefschetz standard conjecture implies the K\"unneth standard conjecture (\cf Remark \ref{RkStand}), it suffices to construct for each $k\in \{0,\cdots,2n\}$, an algebraic cycle $\sZ^\dagger_k\in \CH^{m-r}(Y\times X)_\Q$ such that (\ref{adjointeqn}) is satisfied. Indeed, we could take $\sZ^\dagger=\sum_k\pi^k_X\circ\sZ^\dagger_k\circ\pi_Y^{k+2r}$, where $\circ$ means composition of correspondences (\cf \cite{MR1644323}) and $\pi$ are K\"unneth projectors which are algebraic by assumption.

Now we construct $\sZ_k^\dagger\in \CH^{m-r}(Y\times X)_\Q$. For simplicity, we give the formula in the case that $0\leq k\leq n$ and $0\leq l:=k+2r\leq m$; the other cases follow immediately since we know that the inverse of the hard Lefschetz isomorphism is given by an algebraic correspondence. For any $\alpha\in H^k(X)$, $\beta\in H^l(Y)$, we have:
\begin{eqnarray*}
 (C\alpha,\beta)_{H^l(Y)} = & \int_Y [\sZ]_*(\alpha)\cdot L_Y^{m-l}\circ s_Y(\beta) & (\text{Definition (\ref{polarization})})\\
= & \int_X\alpha\cdot [\sZ]^*\circ L_Y^{m-l}\circ s_Y(\beta) & (\text{projection formula})\\
= & \int_X L_X^{n-k}\alpha\cdot s_X\circ (L_X^{n-k})^{-1}\circ s_X\circ [\sZ]^*\circ L_Y^{m-l}\circ s_Y(\beta) & (\text{Remark \ref{rmk-s}})\\
= & (\alpha, (L_X^{n-k})^{-1}\circ s_X\circ [\sZ]^*\circ L_Y^{m-l}\circ s_Y(\beta))_{H^k(X)} & (\text{Definition (\ref{polarization})})
\end{eqnarray*}
The $s$-operators and the inverse of the Lefschetz operator are supposed to be algebraic by the Lefschetz standard conjecture as the preceding lemma shows. We use the same notation to denote the algebraic cycles inducing them.
Therefore, we can take
\begin{equation}\label{adjoitcycle}
 \sZ^\dagger=(L_X^{n-k})^{-1}\circ s_X\circ {}^{t}\!\sZ\circ L_Y^{m-l}\circ s_Y
\end{equation}
where ${}^{t}\!\sZ\in \CH^{n+r}(Y\times X)_\Q$ is the \emph{transpose} of the correspondence $\sZ\in \CH^{n+r}(X\times Y)_\Q$ (\cf \cite{MR1644323}), and $\circ$ means the composition of correspondences. $C^\dagger$ is defined to be the cohomological correspondence induced by $\sZ^\dagger$.
\end{proof}
\begin{rmk}\upshape
Although the adjoint correspondence $\sZ^\dagger$ of $\sZ$ is not uniquely determined as an algebraic cycle modulo rational equivalence, the adjoint (cohomological) correspondence $C^\dagger$ of $C$ is uniquely determined as a cohomological class in $H^*(Y\times X)$, since the polarization is non-degenerate.
\end{rmk}

As expected, we have immediately:
\begin{lemma}\label{easylemma} Let $X$, $Y$, $r$, $\sZ$, $C$ be as above, then
\begin{itemize}
 \item For any $k\in \{0,\cdots,2n\}$, $\alpha\in H^k(X)$, $\beta\in H^{k+2r}(Y)$, we have $$(C^\dagger\beta,\alpha)_{H^k(X)}=(\beta, C\alpha)_{H^{k+2r}(Y)}.$$ 
 \item The operator $\dagger$ is an \emph{involution}:$$C^{\dagger\dagger}=C.$$
 \item If we have a third smooth projective variety $Z$, and an algebraic correspondence from $Y$ to $Z$: $\sZ'\in \CH(Y\times Z)_\Q$, and let $C'$ be the corresponding cohomological correspondence, then we have a \emph{functoriality}:
$$(C'C)^\dagger=C^\dagger C'^\dagger$$
\end{itemize}

\end{lemma}
\begin{proof}
Indeed,
\begin{eqnarray*}
 (C^\dagger\beta,\alpha)_{H^k(X)}&=&(-1)^k(\alpha, C^\dagger\beta)_{H^k(X)}\\
&=&(-1)^{k+2r}(C\alpha, \beta)_{H^{k+2r}(Y)} \hspace{0.5cm}(\text{by (\ref{adjointeqn})})\\
&=&(\beta, C\alpha)_{H^{k+2r}(Y)}
\end{eqnarray*}
gives the first assertion, and
$$(\beta, C^{\dagger\dagger}\alpha)_{H^{k+2r}(Y)}=(C^\dagger\beta,\alpha)_{H^k(X)}=(\beta, C\alpha)_{H^{k+2r}(Y)}$$
yields the second one by the non-degeneracy of the polarization.
Similarly,
$$(\alpha, (C'C)^\dagger\gamma)=(C'C\alpha, \gamma)=(C\alpha, C'^\dagger\gamma)=(\alpha,C^\dagger C'^\dagger\gamma)$$
gives the third assertion.
\end{proof}

% \begin{rmk}\upshape
% In fact, in our construction of adjoint correspondence (\ref{adjoitcycle}), the first equality in the above lemma is true even in the level of algebraic cycle modulo rational equivalence:
% \begin{eqnarray*}
%  \sZ^{\dagger\dagger} &=& ((L_X^{n-k})^{-1}\circ s_X\circ {}^{t}\!\sZ\circ L_Y^{m-l}\circ s_Y)^\dagger \hspace{5cm}(\text{by (\ref{adjoitcycle})})\\
% &=& (L_Y^{m-l})^{-1}\circ s_Y\circ {}^{t}\!((L_X^{n-k})^{-1}\circ s_X\circ {}^{t}\!\sZ\circ L_Y^{m-l}\circ s_Y) \circ L_X^{n-k}\circ s_X  ~~~~~~~(\text{by (\ref{adjoitcycle})})\\
% &=& (L_Y^{m-l})^{-1}\circ s_Y\circ s_Y\circ L_Y^{m-l}\circ {}^{t t}\!\sZ\circ s_X\circ (L_X^{n-k})^{-1}\circ L_X^{n-k}\circ s_X\\
% &=& \sZ
% \end{eqnarray*}
% where the reason of the last equality is that we could choose $s_X$ in such a way that Remark \ref{rmk-s} holds in the level of algebraic cycles modulo rational equivalence. Since in principle the $\sZ^\dagger$ is not uniquely determined modulo rational equivalence, we do not want to make such a statement.
% \end{rmk}

The following formal property will play an important role in the final part of our argument. It appeared in \cite{MR2082666}, Lemma 5. We recall that the restriction of a polarization on a Hodge structure to a sub-Hodge structure is non-degenerate (\cf Remark \ref{advantage}).

\begin{prop}[Invariance of rank]\label{invrank}
Let $X$, $Y$, $\sZ$, $C$ as above. We have
$$\rank(C)=\rank(C^\dagger)=\rank(CC^\dagger)=\rank(C^\dagger C).$$
In particular, $$\Ker(CC^\dagger)=\Ker(C^\dagger);$$
$$\Ker(C^\dagger C)=\Ker(C);$$
$$\im(CC^\dagger)=\im(C);$$
$$\im(C^\dagger C)=\im(C^\dagger).$$
\end{prop}
\begin{proof}
 We only need to show $\Ker(C^\dagger C)=\Ker(C)$. Indeed, replacing $C$ by $C^\dagger$ gives another equality for kernels since $C^{\dagger\dagger}=C$, then combining the obvious fact that $\rank(C)=\rank(C^\dagger)$ we get all the equalities of ranks, and the equalities of images follow immediately.

Now $\Ker(C^\dagger C)\supset\Ker(C)$ is obvious. For the other inclusion, let $\alpha\in \Ker(C^\dagger C)$, we have $(C\alpha,C\alpha')=(C^\dagger C\alpha,\alpha')=0$ for any $\alpha'\in H^k(X)$. However, since $C$ is induced by an algebraic correspondence, it is a morphism of Hodge structures; in particular $\im(C)$ is a sub-Hodge structure, therefore as we remarked above, the restriction $(-,-)|_{\im(C)}$ is non-degenerate, which implies $C\alpha=0$, \ie $\alpha\in \Ker(C)$. This gives the other inclusion $\Ker(C^\dagger C)\subset\Ker(C)$.
\end{proof}

\subsection{The proof of the main theorem}
We now prove the main theorem which says that Conjecture \ref{mainquestion} is implied by the Lefschetz standard conjecture:
\begin{thm}[= Theorem \ref{mainthm0}]\label{mainthm}
 Let $X$ be a smooth projective variety of dimension $n$, $0\leq k\leq n$ be an integer, and $\gamma\in H^{2n-2k}(X,\Q)$ be a big cohomology class\footnote{\cf Definition \ref{bigness}}. Let $L$ be the kernel of the following morphism of `cup product with $\gamma$':
$$\cup\gamma: H^k(X,\Q)\to H^{2n-k}(X,\Q).$$
Assuming the Lefschetz standard conjecture, then $L$ is supported on a divisor of $X$, that is, $$L\subset \Ker(H^k(X,\Q)\to H^k(X\backslash Z, \Q))$$ for some closed algebraic subset $Z$ of codimension 1.
\end{thm}
\begin{proof}
%  Firstly, let us recall the construction of the `geometrization' of $L$ in the proof of Proposition \ref{geometrization}: assume (by bigness of $\gamma$) $$\gamma=[A]^{n-k}+[E],$$
% where $A=\sO_X(1)$ is a general hyperplane section and $E=\sum_im_iE_i$ with $m_i\in \N_+$ is an effective algebraic cycle of dimension $k$, and $\tau_i:\tilde E_i\to E_i$ be a resolution of singularities for each $i$. Let $B$ be the intersection of $(n-k+1)$ general hyperplane sections, and $H_i$ be a general section of a very ample line bundle on $\tilde E_i$, and $$Y:=B\sqcup\bigsqcup_iH_i$$ be the disjoint union, which is a (non-connected in general) smooth projective variety $Y$ of dimension $k-1$. We have proved the claim (**) there:\\
% the composition $L\inj H^k(X)\lra{\miu^*}H^k(Y)$ is injective, \ie $L\cap \Ker(\miu^*)=\{0\}$.

Firstly, recall that in Proposition \ref{geometrization} we have constructed a (not necessarily connected) smooth projective variety $Y$ of dimension $k-1$ with a morphism $\miu: Y\to X$, and showed that the composition $L\inj H^k(X,\Q)\lra{\miu^*}H^k(Y,\Q)$ is injective, \ie $L\cap \Ker(\miu^*)=\{0\}$.

Note that the Lefschetz standard conjecture on $Y$ tells us the inverse hard Lefschetz isomorphism $(L_Y)^{-1}:H^k(Y,\Q)\lra{\isom}H^{k-2}(Y,\Q)$ is algebraic. Therefore, the composition
\begin{equation}\label{C}
 C:=(L_Y)^{-1}\circ\miu^*: H^k(X,\Q)\to H^{k-2}(Y,\Q)
\end{equation}
is algebraic, \ie $C=[\sZ]_*$ for some $\sZ\in \CH^{n-1}(X\times Y)_\Q$. The above injectivity is of course preserved, thus
$$\Ker(C)\cap L=\{0\}.$$
Taking the orthogonal complements (with respect to the fixed polarization $(-,-)_{H^k(X)}$ introduced in (\ref{polarization})) of both sides, and using the non-degeneracy of the polarization, we get:
\begin{equation}\label{surj1}
 \Ker(C)^{\perp}+L^{\perp}=H^k(X,\Q).
\end{equation}
Now consider the adjoint correspondence $C^\dagger=[\sZ^\dagger]_*: H^{k-2}(Y,\Q)\to H^k(X,\Q)$, where $\sZ^\dagger\in \CH_{n-1}(Y\times X)_\Q$ is the algebraic cycle constructed in (\ref{adjoitcycle}) using the Lefschetz standard conjecture.
By the adjoint property $(C^\dagger\alpha,\alpha')=(\alpha,C\alpha')$ for any $\alpha\in H^{k-2}(Y,\Q)$ and $\alpha'\in H^k(X,\Q)$, we find that $\im(C^\dagger)\subset\Ker(C)^\perp$. However, $\dim\Ker(C)^\perp=\dim H^k(X)-\dim\Ker(C)=\dim\im(C)=\dim\im(C^\dagger)$, so in fact $\im(C^\dagger)=\Ker(C)^\perp$. Therefore (\ref{surj1}) is equivalent to
\begin{equation}\label{surj2}
 \im(C^\dagger)+L^{\perp}=H^k(X,\Q).
\end{equation}
We first finish the proof by assuming the following Proposition \ref{projectorL}, which says that with respect to the orthogonal decomposition $H^k(X,\Q)=L\oplus L^\perp$ (\cf Remark \ref{advantage}), the orthogonal projector $$\pr_L:H^k(X,\Q)\surj L\inj H^k(X,\Q)$$ is algebraic, \ie induced by an algebraic cycle $\sZ'\in \CH^n(X\times X)_\Q$. \\
Now consider the composition of the algebraic correspondences $\sZ'\circ\sZ^\dagger$:
\begin{displaymath}
  \xymatrix{
  H^{k-2}(Y,\Q) \ar[rr]^{C^\dagger=[\sZ^\dagger]_*} \ar@{->>}[dr]
                &  &    H^k(X,\Q)\ar[rr]^{\pr_L=[\sZ']_*} \ar@{->>}[dr] & & H^k(X,\Q) \\
                & \im(C^\dagger)\ar@{^{(}->}[ur] \ar@{->>}[rr] & & L \ar@{^{(}->}[ur]            }
\end{displaymath}
where we place the cohomological correspondences in the first line, and the images of two morphisms in the second line. Then the equality (\ref{surj2}) says exactly that the induced morphism in the bottom line from $\im(C^\dagger)$ to $L$ is surjective, in other words, $$\im\left([\sZ'\circ\sZ^\dagger]_*:H^{k-2}(Y,\Q)\to H^k(X,\Q)\right)=L.$$
Therefore, $L$ is supported on $Z:= \Supp\left(\pr_2(\sZ'\circ\sZ^\dagger)\right)$: the support of the image of $\sZ'\circ\sZ^\dagger\in \CH_{n-1}(Y\times X)_\Q$ under the projection to $X$, so the dimension of each irreducible component of $Z$ is at most $n-1$, hence $L$ is supported on a divisor of $X$.
\end{proof}

To complete the proof, we only need to show the following
\begin{prop}[Orthogonal projector to $L$]\label{projectorL}
Let $X$, $\gamma$, $L$ as in the above theorem. Then for the orthogonal decomposition\footnote{See the last point of Remark \ref{advantage}.} $H^k(X,\Q)=L\oplus L^\perp$ with respect to the fixed polarization $(-,-)_{H^k(X)}$, the orthogonal projector $$\pr_L:H^k(X,\Q)\surj L\inj H^k(X,\Q)$$ is algebraic (in the sense of Remark \ref{morphism}).
\end{prop}
\begin{proof}
 Define $B:H^k(X,\Q)\to H^k(X,\Q)$ to be the unique morphism satisfying
$$(B\alpha, \alpha')_{H^k(X)}=\int_X\gamma\alpha\alpha'$$ for any $\alpha, \alpha'\in H^k(X)$. (Here the rationality of $B$ comes from those of $\gamma$ and $s_X$.)\\
However, by
\begin{eqnarray*}
 \int_X\gamma\alpha\alpha' &=& \int_XL_X^{n-k}\cdot (L_X^{n-k})^{-1}\circ s_X(\gamma\alpha)\cdot s_X(\alpha')\\ &=& ((L_X^{n-k})^{-1}\circ s_X(\gamma\alpha), \alpha'),
\end{eqnarray*}
we deduce that $$B=(L_X^{n-k})^{-1}\circ s_X\circ(\gamma\cup),$$
where $\gamma$ is an algebraic class, $s_X$ and $(L_X^{n-k})^{-1}$ are also induced by algebraic correspondences under the assumption of Lefschetz standard conjecture, therefore $B$ is algebraic, \ie $B=[\sW]_*$, for some $\sW\in\CH^n(X\times X)_\Q$.

Observe that $(B\alpha, \alpha') =\int_X\gamma\alpha\alpha'=(-1)^k\int_X\gamma\alpha'\alpha=(-1)^k(B\alpha', \alpha) =(\alpha,B\alpha')$, which means that $B$ is \emph{self-adjoint}: $$B^\dagger=B.$$
By Proposition \ref{invrank}, we have $\rank(B^2)=\rank(BB^\dagger)=\rank(B)$, \ie,
\begin{equation}\label{B}
 \im(B)=\im(B^2).
\end{equation}
The following elementary lemma in linear algebra allows us to construct from such an endomorphism a projector onto its image.

\begin{lemma}
 Let $V$ be a finite dimensional $\Q$-vector space, $f:V\to V$ be an endomorphism satisfying $\im(f^2)=\im(f)$. Then there exists a $\Q$-coefficient polynomial $P$ with $P(0)=0$, such that the endomorphism $g:=P(f)$ is a projector onto $\im(f)$, \ie $g^2=g$ and $\im(g)=\im(f)$. Moreover, $\Ker(f)=\Ker(g)$.
\end{lemma}
\begin{proof}
 By assumption $f|_{\im(f)}:\im(f)\to \im(f)$ is surjective hence an isomorphism. Let $Q\in\Q[T]$ be the minimal polynomial of $f|_{\im(f)}$, since $f|_{\im(f)}$ is an isomorphism, $Q(0)\neq 0$.\\
Defining $R\in\Q[T]$ to be $R[T]=-\frac{Q(T)-Q(0)}{Q(0)\cdot T}$, then
$R\left(f|_{\im(f)}\right)=\left(f|_{\im(f)}\right)^{-1}$; in other words,
 $$\left(R(f)\cdot f\right)|_{\im(f)}=\id_{\im(f)}.$$
Now we set $P\in \Q[T]$ to be $P(T)=R(T)\cdot T$. Then $P(0)=0$, and $g:=P(f)$ satisfies
\begin{equation}\label{g}
 g|_{\im(f)}=\id_{\im(f)}.
\end{equation}
However, since $P(0)=0$, we have $\im(g)\subset\im(f)$, thus (\ref{g}) implies $\im(g)=\im(f)$ and thus also $g^2=g$, \ie $g$ is a projector onto $\im(f)$.\\
Moreover, by $P(0)=0$, we have \apriori $\Ker(f)\subset\Ker(g)$; but $f$ and $g$ have the same image, thus the same rank, we deduce that $\Ker(f)=\Ker(g)$.
\end{proof}

We continue the proof of the Proposition. By (\ref{B}), we can apply the above lemma to $B$ to get a rational coefficient polynomial $P$ with $P(0)=0$, such that $P(B)$ is a projector onto $\im(B)$, and $\Ker\left(P(B)\right)=\Ker(B)$. Therefore, $P(B)$ and $\id-P(B)$ is a pair of projectors corresponding to the direct sum decomposition $$H^k(X,\Q)=\im(B)\oplus\Ker(B).$$ Moreover, we remark that the above direct sum decomposition is in fact \emph{orthogonal} with respect to $(-,-)_{H^k(X)}$: this is an immediate consequence of the self-adjoint property of $B$. \\
To conclude, we remark that $L=\Ker(\gamma\cup)=\Ker(B)$, thus $\Delta_X-P(\sW)\in \CH^n(X\times X)_\Q$ induces on the cohomology $H^k(X,\Q)$ the orthogonal projector $\id-P(B)$ onto $L$, where $\Delta_X$ denotes the diagonal class in $X\times X$, and the multiplication in $P(\sW)$ is given by composition of correspondences (NOT the intersection product).

This finishes the proof of the Proposition \ref{projectorL} and thus also the proof of the main Theorem \ref{mainthm}.
\end{proof}

\section{Final Remarks}
\begin{rmk}[Unconditional results]\upshape
Our proof of Conjecture \ref{mainquestion} using the standard conjecture is in fact unconditional in some cases. In the following discussion, let $X$, $\gamma$, $L$ be as in the main theorem \ref{mainthm}, and we adopt all the constructions and notations of its proof in the preceding section.

When $k=0, 1$, there is nothing to prove. The $k=2$ case reduces to the Lefschetz theorem on (1,1)-classes for $H^2(X,\Q)$.

When $k=3, 4, 5$, recall that the correspondence needed in the proof of the main theorem \ref{mainthm} is $\pr_L\circ C^\dagger: H^{k-2}(Y,\Q)\to H^k(X,\Q)$, and we use Lefschetz standard conjecture on $Y$ to get the algebraicity of $C^\dagger$, and use it on $X$ to get the algebraicity of $\pr_L$. However, by an explicit calculation:
\begin{eqnarray*}
 C^\dagger =&((L_Y)^{-1}\circ \miu^*)^\dagger &~~~~~~ (\text{see (\ref{C})})\\
=& (\miu^*)^\dagger\circ (L_Y^{-1})^\dagger & ~~~~~ (\text{Lemma \ref{easylemma}})\\
=& ((L_X^{n-k})^{-1}\circ s_X\circ \miu_*\circ s_Y\circ L_Y)\circ L_Y^{-1} & ~~~~~ (\text{by (\ref{adjoitcycle})})\\
=& (L_X^{n-k})^{-1}\circ s_X\circ \miu_*\circ s_Y
\end{eqnarray*}
we find that we only need the standard conjecture on $X$ and the algebraicity of the morphism $s_Y: H^{k-2}(Y,\Q)\to H^{k-2}(Y,\Q)$. While the algebraicity of $s_Y$ on $H^i(Y)$ for $i\leq 3$ is known: firstly $s_Y$ acts as identity on $H^0(Y)$ and $H^1(Y)$, thus is obviously algebraic; as for $H^2(Y)$ (\emph{resp.} $H^3(Y)$), the Lefschetz decomposition has only two factors, and the projector to the primitive factor can be constructed using only the Lefschetz operator and the inverse $H^{2\dim Y}(Y)\lra{\isom} H^0(Y)$ (\emph{resp.} $H^{2\dim Y-1}(Y)\lra{\isom} H^1(Y)$), which is also known to be algebraic. In conclusion, our proof works unconditionally when $k=3, 4, 5$ for $X$ a smooth complete intersection of a product of curves, surfaces, and abelian varieties \etc.
\end{rmk}
\begin{rmk}[A reinterpretation by motivated cycles]\upshape
To get around the standard conjectures and thus obtain some unconditional theories of motives, Y. Andr\'e \cite{MR1423019}  introduced the notion of \emph{motivated cycles}, which is a space of cohomology classes fitting in the following inclusions (conjecturally they are all the same):\\
\{classes of cycles\}$\subset$\{motivated cycles\}$\subset$\{absolute Hodge classes\}$\subset$\{Hodge classes\}.\\
Roughly speaking, motivated cycles are constructed from algebraic cycles by adding the cohomology classes of the inverses of hard Lefschetz isomorphisms in the category of smooth projective varieties with morphisms given by algebraic correspondences. We refer to the original paper \loccit for more details, and also to \cite{MR2115000} Chapter 9, 10 for an introduction.

Now if we considered motivated cycles and motivated correspondences (=motivated cycles in the product spaces) instead of the algebraic ones, we would not have any problem caused by the standard conjectures. In particular, we could define a sub-Hodge structure $L$ of $H^k(X,\Q)$ to be of \emph{motivated coniveau} at least $c$ if there exists a motivated correspondence $\Gamma$ from another smooth projective variety $Y$ to $X$, such that $L$ is contained in the image of $\Gamma_*: H^{k-2c}(Y)\to H^k(X)$. In this language, our result Theorem \ref{mainthm} can be reformulated as:

\begin{thm}
 Let $X$ be a smooth projective variety of dimension $n$, $0\leq k\leq n$ be an integer, and $\gamma\in H^{2n-2k}(X,\Q)$ be a big cohomology class. Let $L$ be the kernel of the following morphism of `cup product with $\gamma$':
$$\cup\gamma: H^k(X,\Q)\to H^{2n-k}(X,\Q).$$
Then $L$ is of motivated coniveau at least 1.
\end{thm}
\end{rmk}

\bibliographystyle{alpha}
\bibliography{biblio_fulie}

\end{document}